\documentclass[12pt]{amsart}

\usepackage[colorlinks=true]{hyperref}
\usepackage{amssymb}
\usepackage[all]{xy}
\usepackage{enumerate}
\usepackage{graphicx}

\newcommand{\G}{\mathbf{G}}
\newcommand{\Z}{\mathbf{Z}}

\newcommand{\cF}{\mathcal{F}}
\newcommand{\cG}{\mathcal{G}}
\newcommand{\cK}{\mathcal{K}}
\newcommand{\cO}{\mathcal{O}}
\newcommand{\cT}{\mathcal{T}}

\newcommand{\CH}{\mathrm{CH}}
\renewcommand{\d}{\mathrm{d}}
\newcommand{\ev}{\mathrm{ev}}
\newcommand{\gr}{\mathrm{gr}}
\newcommand{\Hom}{\mathrm{Hom}}
\newcommand{\id}{\mathrm{id}}
\newcommand{\Map}{\mathrm{Map}}
\newcommand{\Pic}{\mathrm{Pic}}
\newcommand{\pt}{\mathrm{pt}}
\newcommand{\Spec}{\mathrm{Spec}}
\newcommand{\Sym}{\mathrm{Sym}}

\newcommand{\B}{\mathrm{B}}
\newcommand{\mC}{\mathrm{C}}
\newcommand{\mH}{\mathrm{H}}
\newcommand{\K}{\mathrm{K}}
\renewcommand{\L}{\mathrm{L}}

\newcommand{\mA}{\mathrm{A}}
\newcommand{\mB}{\mathrm{B}}
\newcommand{\mD}{\mathrm{D}}
\newcommand{\mE}{\mathrm{E}}
\newcommand{\mF}{\mathrm{F}}
\newcommand{\mG}{\mathrm{G}}
\newcommand{\BDG}{\mathrm{BDG}}

\newcommand{\llpar}{(\!(}
\newcommand{\rrpar}{)\!)}

\newtheorem{lm}{Lemma}[section]
\newtheorem{prop}{Proposition}[section]
\newtheorem{thm}{Theorem}[section]
\newtheorem*{thm*}{Theorem}

\theoremstyle{remark}
\newtheorem*{remark}{Remark}

\makeatletter
\newcommand{\proofstep}[1]{%
  \par
  \addvspace{\medskipamount}
  \textit{#1\@addpunct{.}}\enspace\ignorespaces
}
\makeatother

\begin{document}
\title{Gerbal central extensions of reductive groups by $\cK_3$}
\address{Mathematical Institute, Radcliffe Observatory Quarter, Woodstock
Road, Oxford, OX2 6GG, UK}
\email{safronov@maths.ox.ac.uk}
\author{Pavel Safronov}
\begin{abstract}
We classify central extensions of a reductive group $G$ by $\cK_3$ and $\B\cK_3$, the sheaf of third Quillen $\K$-theory groups and its classifying stack. These turn out to be parametrized by the group of Weyl-invariant quadratic forms on the cocharacter lattice valued in $k^\times$ and the group of integral Weyl-invariant cubic forms on the cocharacter lattice respectively.
\end{abstract}
\maketitle

\section*{Introduction}

Let $G$ be a reductive group over a field and let $\L G=G \llpar t\rrpar$ be its loop group. In representation theory one is often interested in central extensions of $\L G$. A central extension of $\L G$ by $\G_m$ is the same as a multiplicative torsor \cite[Exp. 7, Section 1]{SGA7} on $\L G$, i.e. a torsor $\cT$ together with an identification $\cT_{xy}\cong \cT_x\otimes \cT_y$ for any two points $x,y\in\L G$. Given a geometric object on the loop space $\L G$ one can wonder whether it comes from the group $G$ itself as $G$ is a much easier geometric object than $\L G$: it is an affine scheme of finite type in contrast to $\L G$ being an ind-scheme.

Indeed, there is a transgression map \cite{KV}
\[\mH^1(G, \cK_2)\rightarrow \mH^1(\L G, \cO^\times),\]
where $\cK_2$ is the sheaf of second Quillen $\K$-theory groups, which sends a $\cK_2$-torsor on $G$ to a $\G_m$-torsor on $\L G$. In fact, it behaves well with respect to multiplicative structures and so it sends a $\cK_2$-central extension of $G$ to a $\G_m$-central extension of $\L G$.

Thus, one is naturally led to consider $\cK_2$-central extensions of $G$. Such a classification has been performed by Esnault--Kahn--Levine--Viehweg \cite{EKLV} and Brylinski--Deligne \cite{BrD}.

Brylinski and Deligne construct a spectral sequence which computes the $\cK_n$-cohomology of $G$ from the $\cK_n$-cohomology of its maximal torus, a much simpler task. Another spectral sequence allows one to classify $\cK_n$-central extensions of $G$ from the knowledge of $\cK_n$-cohomology of $G$.

For $n=2$ both spectral sequences degenerate after the first nontrivial page. This allowed Brylinski and Deligne to obtain the following classification:
\begin{thm*}[Brylinski--Deligne]
Let $G$ be a split connected reductive group over a field with a simply-connected derived group. Then the group of $\cK_2$-central extensions of $G$ is isomorphic to the group of Weyl-invariant quadratic forms on the cocharacter lattice.
\end{thm*}

Studying double loop groups $\L^2 G=G\llpar t\rrpar\llpar s\rrpar$, one naturally encounters not ordinary central extensions, but gerbal extensions, i.e. central extensions by the commutative group stack $\B\G_m$ \cite{AK}, \cite{FZ}. Since we have to do the transgression map twice, $\G_m$-central extensions of $\L^2 G$ come from $\cK_3$-central extensions of $G$. The goal of this paper is to classify such extensions. We obtain the following result (Theorem \ref{thm:main}).
\begin{thm*}
Let $G$ be a split connected reductive group over a field $k$ with a simply-connected derived group. Then the group of $\cK_3$-central extensions of $G$ is isomorphic to the group of Weyl-invariant quadratic forms on the cocharacter lattice valued in $k^\times$. The group of gerbal central extensions of $G$ by $\cK_3$ is isomorphic to the group of integral Weyl-invariant cubic forms on the cocharacter lattice.
\end{thm*}

We show that using our methods one can compute $\CH^3(G)$, the third Chow group of codimension 3 cycles on the group $G$. It turns out to be isomorphic to $(\Z/2\Z)^{n_{\BDG}}$, where $n_{\BDG}$ is the number of times a $\mG_2$, $\mB_3$ or $\mD_4$ subdiagram appears in the Dynkin diagram of $G$. Thus, $\CH^3(G)$ is 2-torsion unless $G$ only contains type $\mA$ and $\mC$ subgroups in which case it is trivial. Note that one has $\CH^1(G) \equiv \Pic(G) = 0$ and $\CH^2(G) = 0$ for all $G$ satisfying our assumptions. The full Chow ring of algebraic groups was known before due to results of Borel, Chevalley, Grothendieck \cite{Gr}, Kac \cite{Kac}, Marlin \cite{Mar}, Kaji--Nakagawa \cite{KN} and many others.

We have concentrated on the simply-connected case for simplicity; the general case has both combinatorial difficulties and the $E_1$ degeneration of the Brylinski--Deligne spectral sequence is not immediate. The results generalize to split groups over more general regular bases over a field, but a generalization to the non-split case is unclear. Our statements could be useful for a partial computation of $\cK$-cohomology $\mH^\bullet(G, \cK_n)$ for $n > 3$.

The paper is organized in the following way. We begin with a recollection on $\cK$-cohomology of tori computed in \cite{EKLV} and \cite{BrD}. We recall a classification of $\cK_2$-classes on tori via alternating forms together with quadratic refinements and prove a version of this classification for $\cK_3$-classes on tori. Section \ref{sect:kcohomologysemisimple} contains the core technical results. We recall the spectral sequence introduced by Brylinski and Deligne for the computation of $H^\bullet(G, \cK_n)$ and compute it for $n=2,3$ for a semisimple group $G$. The spectral sequence degenerates at $E_1$ in both cases. In Section \ref{sect:kcohomologyreductive} we combine the results of Sections \ref{sect:kcohomologytori} and \ref{sect:kcohomologysemisimple} to obtain $\cK_3$-cohomology of reductive groups $G$. In the final section we use these results to obtain a classification of multiplicative $\cK_3$-torsors and gerbes on the group and thus the corresponding central extensions. We end with an explanation of the relation between $\cK_n$-extensions of $G$ and $\G_m$-extensions of its iterated loop groups.

\subsection*{Acknowledgements}
The author would like to thank Xinwen Zhu for a useful conversation, Burt Totaro for some comments on the draft and the referee for improving the exposition.

\section{$\cK$-cohomology of tori}
\label{sect:kcohomologytori}

Let $H$ be a split torus over a field $k$. We denote by $X=\Hom(H, \G_m)$ its group of characters and by $Y=\Hom(\G_m, H)$ the dual group of cocharacters. Let $\cK_n$ be the Zariski sheafification of the presheaf of Quillen $\K$-theory groups. For instance, $\cK_1\cong\cO^\times$ and we denote by $\{f\}$ the element of $\cK_1$ corresponding to the invertible function $f\in \cO^\times$. The sheaves $\cK_\bullet$ form a graded ring and we denote the multiplication in $\K$-theory by a period.

We recall a computation of the $\cK$-cohomology of $H$ due to Brylinski and Deligne \cite[Lemma 3.3.1]{BrD}:

\begin{prop}[Brylinski--Deligne]
The graded ring $\mH^0(H, \cK_\bullet)$ is generated by the graded ring $\K_\bullet(k)$ and $X$ in degree 1 subject to the following relations:
\begin{enumerate}
\item $X\rightarrow \mH^0(H, \cK_1)$ is additive,
\item $x.x = x.\{-1\}$ for every $x\in X$.
\end{enumerate}

Moreover, the higher cohomology vanishes.
\end{prop}

From the presentation we see that there is an increasing filtration $V$ on $\mH^0(H, \cK_\bullet)$ compatible with the grading where $V_n \mH^0(H, \cK_\bullet)$ is the subgroup of elements of $\mH^0(H, \cK_\bullet)$ containing not more than $n$ factors of $X$. The associated graded ring is given by
\[\gr_n^V\mH^0(H, \cK_m) = \bigwedge^n X\otimes \K_{m-n}(k).\]

Note that $\K_m(k)\subset \mH^0(H, \cK_m)$ is a direct summand, so we also have an embedding $X\otimes \K_{m-1}(k)\hookrightarrow \mH^0(H, \cK_m)$ given by the multiplication in $\K$-theory.

The filtration on $\mH^0(H, \cK_2)$ has a direct summand $\K_2(k)$ and thus the reduced cohomology group $\tilde{\mH}^0(H, \cK_2)$ fits into an exact sequence
\[0\rightarrow X\otimes k^\times \rightarrow \tilde{\mH}^0(H, \cK_2) \rightarrow \wedge^2 X\rightarrow 0.\]

This extension has the following description.

\begin{prop}[Brylinski--Deligne]
The reduced cohomology group $\tilde{\mH}^0(H, \cK_2)$ is isomorphic to the group of alternating forms $A$ on the cocharacter lattice together with a quadratic refinement $q\colon Y\rightarrow k^\times$ obeying
\[\frac{q(x+y)}{q(x)q(y)} = (-1)^{A(x, y)}.\]
\label{prop:k2extension}
\end{prop}

Note that $q\colon Y\rightarrow k^\times$ is not a map of groups unless $A$ is zero. We will also need an analog of this proposition for $\cK_3$-classes.

\begin{prop}
\label{prop:k3extension}
The reduced cohomology $\tilde{\mH}^0(H, \cK_3)$ is isomorphic to the group of totally antisymmetric 3-forms $A(-, -, -)$ on the cocharacter lattice together with two refinements $q_1(-, -)\colon Y\times Y\rightarrow k^\times$ and $q_2(-)\colon Y\rightarrow \K_2(k)$ satisfying
\begin{align*}
q_1(x, y) &= q_1(y, x)^{-1}, \\
q_1(x + y, z) &= q_1(x, z) q_1(y, z) (-1)^{A(x, y, z)}, \\
q_2(x + y) &= q_2(x).q_2(y).\{-1, q_1(x, y)\}.
\end{align*}
\end{prop}
\begin{proof}
If $\{e_i\}$ denotes a basis of $X$, a general element of $\tilde{\mH}^0(H, \cK_3)$ can be represented as
\[a_{ijk} e_i.e_j.e_k + e_i.e_j.\{f_{ij}\} + e_i.g_i\]
for some integer matrix $a_{ijk}$, a collection $f_{ij}$ of elements of $k^\times$ and a collection $g_i$ of elements of $\K_2(k)$.

We define \[a\in X^{\otimes 3},\quad f\in X\otimes X\otimes k^\times,\quad g\in X\otimes \K_2(k)\]
by a linear extension of these collections.

Any two representations of the same element are connected by a sequence of the following three identifications:
\begin{equation}
\begin{aligned}
a_{ijk}&\sim a_{ijk} + c_ic_j,\quad f_{ij}\sim f_{ij}(-1)^{c_i},\quad g_i\sim g_i, \\
a_{ijk}&\sim a_{ijk} + c_jc_k,\quad f_{ij}\sim f_{ij}(-1)^{c_j},\quad g_i\sim g_i, \\
a_{ijk}&\sim a_{ijk},\quad f_{ij}\sim f_{ij}h^{c_ic_j},\quad g_i\sim g_i.\{(-1)^{c_i}, h\}.
\end{aligned}
\label{k3choice}
\end{equation}
Here $c_i$ is a collection of integer numbers and $h\in k^\times$.

We define
\[A(x, y, z) = a(x, y, z) - a(y, x, z) - a(x, z, y) + a(y, z, x) + a(z, x, y) - a(z, y, x),\]
the antisymmetrization of $a$. It is clearly well-defined for an element of $\tilde{\mH}^0(H, \cK_3)$.

Next, we define
\[q_1(x, y) = \frac{f(x, y)}{f(y, x)}(-1)^{a(x, x, y) + a(x, y, x) + a(y, x, x) + a(x, y, y) + a(y, x, y) + a(y, y, x)}.\]
Again, a straightforward check shows that it is invariant under transformations \eqref{k3choice}.

The form $q_1$ is obviously alternating. It is not linear in the arguments, but satisfies the following equations:
\begin{align*}
q_1(x_1 + x_2, y) &= q_1(x_1, y) q_1(x_2, y) (-1)^{A(x_1, x_2, y)}, \\
q_1(x, y_1 + y_2) &= q_1(x, y_1) q_1(x, y_2) (-1)^{A(x, y_1, y_2)}.
\end{align*}

Finally, we define
\[q_2(x) = g(x).\{-1, f(x, x)\}.\]

It satisfies the following equation:
\begin{align*}
q_2(x+y) &= g(x+y).\{-1, f(x+y, x+y)\} \\
&= g(x).g(y).\{-1, f(x, x)f(x, y)f(y, x)f(y, y)\} \\
&= q_2(x)q_2(y).\{-1, f(x, y)f(y, x)\} \\
&= q_2(x).q_2(y).\{-1, q_1(x, y)\}.
\end{align*}

Let us now prove that the data of $A$, $q_1$ and $q_2$ satisfying the linearity equations above is equivalent to specifying an element of $\tilde{\mH}^0(H, \cK_3)$.

Indeed, we have defined a surjective map $\tilde{\mH}^0(H, \cK_3)\twoheadrightarrow \wedge^3 X$. Its kernel $K$ is represented by elements with $a_{ijk}$ whose antisymmetrization is zero. Such a matrix can be made zero using the first two transformations in \eqref{k3choice}. So, elements of $K$ can be represented as
\[e_i.e_j.\{f_{ij}\} + e_i.g_i.\]

We have defined a surjective map $K\twoheadrightarrow \wedge^2 X\otimes k^\times$. Its kernel consists of elements with $f_{ij}$ symmetric. Using the last transformation in \eqref{k3choice} we make $f_{ij} = 1$, so we get an exact sequence
\[0\rightarrow X\otimes \K_2(k)\rightarrow K\rightarrow \wedge^2 X\otimes k^\times\rightarrow 0.\]

Note that the filtration
\[X\otimes \K_2(k)\subset K\subset \tilde{\mH}^0(H, \cK_3)\]
we have defined during the proof coincides with the $V$ filtration.
\end{proof}

\section{$\cK$-cohomology of semisimple groups}

\label{sect:kcohomologysemisimple}

In this section we let $G$ be a split connected simply-connected semisimple group over a field $k$, $H\subset G$ is a split maximal torus and $W$ the Weyl group. Let $W^{(n)}\subset W$ be the subset of elements of $W$ of length $l(w_0) - n$, where $l(w_0)$ is the length of the longest element. Let $X$ be the character group of $H$ and $Y$ the cocharacter group. Let $s_i\in W$ denote the simple reflections with respect to a chosen Borel in $G$ and $\alpha_i^\vee\in Y$ are the simple coroots. Since $G$ is simply-connected, the weight lattice coincides with $X$.

\subsection{Brylinski--Deligne spectral sequence}

The Bruhat decomposition of $G$ allows one to compute the cohomology $\mH^\bullet(G, \cK_n)$ via the Cousin complex $\mC^\bullet(G, \cK_n)$, where each term is expressed in terms of $\mH^0(H, \cK_j)$ for some $j$. The filtration on $\mH^0(H, \cK_j)$ extends to a filtration on $\mC^\bullet(G, \cK_n)$. Thus, Brylinski and Deligne arrive at a spectral sequence of the filtered complex whose $E_0$ page has terms
\[E_0^{p, q} = \bigoplus_{w\in W^{(p+q)}} \bigwedge^{-2p - q} X\otimes \K_{n+p}(k).\]

Note that since the filtration on $\mH^0(H, \cK_j)$ is increasing, the cohomological spectral sequence is concentrated in the second quadrant.

The $E_0$ differential $\d\colon E_0^{p, q}\rightarrow E_0^{p, q+1}$ is given by the Bruhat order and has the following description. It is enough to specify the differential $\d$ by its components
\[\d_w^{w_1}\colon \bigwedge^{-2p-q} X\otimes \K_{n+p}(k)\rightarrow \bigwedge^{-2p-q-1} X\otimes \K_{n+p}(k),\]
where the first term sits at $w\in W^{(p+q)}$ and the second term sits at $w_1\in W^{(p+q+1)}$. If $w=w'w''$ and $w_1=w's_iw''$ for a simple reflection $s_i$, then the differential $\d_w^{w_1}$ is given by contracting the $\wedge^{-2p-q} X$ part with $(w'')^{-1}(\alpha_i^\vee)$.

The spectral sequence converges to the $\cK$-cohomology of $G$:
\[E_0^{p, q}\Rightarrow \mH^{p+q}(G, \cK_n).\]

\subsection{$\cK_2$-cohomology}

Let us illustrate the power of the spectral sequence by computing the $\cK_2$-cohomology of $G$. We will outline the proof of the following theorem \cite[Proposition 4.6]{BrD}:
\begin{thm}[Brylinski--Deligne]
Let $G$ be a split connected simply-connected semisimple group over a field $k$. It has the following $\cK_2$-cohomology groups:
\begin{itemize}
\item $\mH^0(G, \cK_2) = \K_2(k)$,
\item $\mH^1(G, \cK_2)$ is isomorphic to the group of $W$-invariant quadratic forms on the cocharacter lattice.
\end{itemize}

The higher cohomology groups vanish.
\end{thm}

The set $W^{(0)}$ has only one element, the longest element $w_0$. The set $W^{(1)}$ can be identified with the set of simple roots as $w_0s_i\in W^{(1)}$ for any $s_i$ a simple reflection. Finally, $W^{(2)}$ consists of elements $w_0s_is_j$ for any two unequal simple roots $s_i$ and $s_j$; the elements $w_0s_is_j$ and $w_0s_js_i$ are identified if $\alpha_i$ and $\alpha_j$ are orthogonal.

The $E_0$ page is
\[
\xymatrix@R-1pc@C-2pc{
\oplus_{W^{(2)}} \Z & 0 & 0 \\
\oplus_{W^{(1)}} X \ar[u] & 0 & 0 \\
\bigwedge^2 X \ar[u] & \oplus_{W^{(1)}} k^\times & 0 \\
0 & X\otimes k^\times \ar[u] & 0 \\
0 & 0 & \K_2(k)
}
\]

Here $p$ labels the horizontal axis and $q$ the vertical axis. The whole spectral sequence is concentrated in the second quadrant with $E_0^{0, 0} = \K_2(k)$.

We can identify $\oplus_{W^{(1)}} \Z$ with the weight lattice: since $G$ is simply-connected, the map $X\rightarrow \oplus_{W^{(1)}} \Z$ sending $x\mapsto \{x(\alpha_i^\vee)\}$ is an isomorphism. Therefore, $E_1^{-1, 1} = E_1^{-1, 2} = 0$. The same argument shows that the differential $E_0^{-p, p}\rightarrow E_0^{-p, p+1}$ is injective, hence $E_1^{-p, p} = 0$.

\begin{prop}
The differential $E_0^{-2, 3}\rightarrow E_0^{-2, 4}$ is surjective.
\label{prop:2rowsurjective}
\end{prop}
\begin{proof}
See the second half of the proof of \cite[Proposition 4.6]{BrD}.
\end{proof}

Therefore, $E_1^{-2, 4} = 0$. Finally, we can identify $E_1^{-2, 3}$ with the group of bilinear forms $C$ on $Y$ satisfying
\[C(\alpha_i^\vee, \alpha_j^\vee) + C(\alpha_j^\vee, s_j(\alpha_i^\vee)) = 0,\]
or, equivalently, with the group of $W$-invariant quadratic forms $\{y\mapsto C(y, y)\}$ on the cocharacter lattice.

The $E_1$ page of the spectral sequence is
\begin{center}
\[
\xymatrix@R-2pc@C-2pc{
0 & 0 & 0 \\
E_1^{-2, 3} & 0 & 0 \\
0 & 0 & 0 \\
0 & 0 & 0 \\
0 & 0 & \K_2(k).
}
\]
\end{center}

The spectral sequence degenerates at $E_1$, thus we obtain the desired cohomology groups.

\subsection{$\cK_3$-cohomology}

Let us apply the spectral sequence for the computation of the cohomology groups $\mH^n(G, \cK_3)$. We will prove the following theorem:
\begin{thm}
Let $G$ be a split connected simply-connected semisimple group over a field $k$. Then it has the following $\cK_3$-cohomology:
\begin{itemize}
\item $\mH^0(G, \cK_3)\cong \K_3(k)$,

\item $\mH^1(G, \cK_3)$ is isomorphic to the group of $W$-invariant quadratic forms on the cocharacter lattice with values in $k^\times$,

\item $\mH^2(G, \cK_3)$ is isomorphic to the group of $W$-invariant cubic forms on the cocharacter lattice,

\item $\CH^3(G)\cong \mH^3(G, \cK_3)$ is isomorphic to $(\Z/2\Z)^{n_{\BDG}}$, where $n_{\BDG}$ is the number of times a $\mG_2$, $\mB_3$ or $\mD_4$ subdiagram appears in the Dynkin diagram of $G$.
\end{itemize}

The higher cohomology groups vanish.
\end{thm}

The groups $\mH^0(G, \cK_n)$ and $\mH^1(G, \cK_n)$ for all $n$ have been previously computed by Gille \cite[Theorem B]{Gil}.

We will need a description of $W^{(3)}$. Its elements are $w_0s_is_js_k$, where the length of $s_is_js_k$ is 3. The following statement can be obtained from the presentation of the Weyl group as a Coxeter group.

\begin{lm}
The product of simple reflections $s_is_js_k$ has length 1 iff any of the following is satisfied:
\begin{itemize}
\item The adjacent indices are equal.

\item The adjacent roots are orthogonal and $i = k$.
\end{itemize}

Two length 3 elements $s_is_js_k$ and $s_ls_ms_n$ are equal iff any of the following holds:
\begin{itemize}
\item $(ijk)$ is obtained from $(lmn)$ by a permutation of orthogonal roots

\item The roots $i$ and $j$ are connected by a single edge and $i=k=m$, $j=l=n$.
\end{itemize}
\end{lm}

The spectral sequence computing $\mH^\bullet(G, \cK_3)$ has the following $E_0$ page:
\begin{center}
\[
\xymatrix@R-1pc@C-2pc{
\oplus_{W^{(3)}} \Z & 0 & 0 & 0 \\
\oplus_{W^{(2)}} X \ar[u] & 0 & 0 & 0 \\
\oplus_{W^{(1)}} \bigwedge^2 X \ar[u] & \oplus_{W^{(2)}} k^\times & 0 & 0\\
\bigwedge^3 X \ar[u] & \oplus_{W^{(1)}} X\otimes k^\times \ar[u] & 0 & 0 \\
0 & \bigwedge^2 X\otimes k^\times \ar[u] & \oplus_{W^{(1)}} \K_2(k) & 0 \\
0 & 0 & X\otimes \K_2(k) \ar[u] & 0 \\
0 & 0 & 0 & \K_3(k)
}
\]
\end{center}

Due to simply-connectivity of $G$ we again have
\[E_1^{-1, 1} = E_1^{-1, 2} = 0\]
and the differentials $E_0^{-2, 2}\rightarrow E_0^{-2, 3}$ and $E_0^{-3, 3}\rightarrow E_0^{-3, 4}$ are injective.

As explained in the previous section, $E_1^{-2, 4} = 0$ and $E_1^{-2, 3}$ is isomorphic to the group of $W$-invariant quadratic forms on the cocharacter lattice with values in $k^\times$.

\begin{prop}
The cohomology group $E_1^{-3, 4}$ is zero.
\label{prop:3colvanishing}
\end{prop}
\begin{proof}
We will represent elements of $\oplus_{W^{(1)}} \wedge^2 X$ by antisymmetric forms $D_i(-, -)$ on $Y$. The closed elements in $E_0^{-3, 4}$ are antisymmetric forms $D_i$ that satisfy
\[D_i(\alpha_j^\vee, \alpha_k^\vee) + D_j(s_j(\alpha_i^\vee), \alpha_k^\vee) = 0\]
for every coroots $\alpha_i^\vee, \alpha_j^\vee, \alpha_k^\vee$. The form $D_i(-, -)$ is exact if its trilinear extension $D(-, -, -)$ is totally antisymmetric.

We expand the closedness equation as
\[D_i(\alpha_j^\vee, \alpha_k^\vee) + D_j(\alpha_i^\vee, \alpha_k^\vee) - \alpha_j(\alpha_i^\vee)D_j(\alpha_j^\vee, \alpha_k^\vee) = 0.\]

Therefore, $D_i(\alpha_j^\vee, \alpha_k^\vee) = -D_j(\alpha_i^\vee, \alpha_k^\vee)$ if $\alpha_i$ and $\alpha_j$ are orthogonal.

To show that $D(-, -, -)$ is antisymmetric in the first two entries for any arguments we just need to show that $D_j(\alpha_j^\vee, \alpha_k^\vee) = 0$ for any $\alpha_j^\vee$ and $\alpha_k^\vee$.

Indeed,
\begin{align*}
D_j(\alpha_j^\vee, \alpha_k^\vee) &= -D_j(\alpha_k^\vee, \alpha_j^\vee) \\
&= D_k(\alpha_j^\vee, \alpha_j^\vee) - \alpha_j(\alpha_k^\vee)D_j(\alpha_j^\vee, \alpha_j^\vee) = 0.
\end{align*}

Here in the first line we have used the antisymmetry of $D_j$ and in the second line we have used the closedness condition.

Therefore, any closed form $D$ is antisymmetric in the first two and the last two arguments, hence it is totally antisymmetric.
\end{proof}

\begin{prop}
The cohomology $E_1^{-3, 6}$ is $(\Z/2\Z)^{n_{\BDG}}$, where $n_{\BDG}$ is the number of times a subdiagram of type $\mG_2$, $\mB_3$ or $\mD_4$ appears in the Dynkin diagram of $G$.
\end{prop}
\begin{proof}
Suppose that one is given a collection of integers $\{d_{ijk}\}\in E_0^{-3, 6}$, where $i, j, k$ parametrize the coroots, defined for unequal adjacent indices satisfying
\begin{itemize}
\item $d_{ijk} = d_{jik}$ if $i$ is orthogonal to $j$,
\item $d_{ijk} = d_{ikj}$ if $j$ is orthogonal to $k$,
\item $d_{iji} = d_{jij}$ if $i$ and $j$ are connected by a single edge.
\end{itemize}

Note, that $d_{iji}$ is undefined if $\alpha_i^\vee$ and $\alpha_j^\vee$ are orthogonal.

The collection $\{d_{ijk}\}\in E_0^{-3, 6}$ is exact if the system
\begin{equation}
d_{ijk} = \phi_{ij}(\alpha_k^\vee) + \phi_{ik}(\alpha_j^\vee - \alpha_k(\alpha_j^\vee) \alpha_k^\vee) + \phi_{jk}(\alpha_i^\vee - \alpha_j(\alpha_i^\vee)\alpha_j^\vee - \alpha_k(\alpha_i^\vee)\alpha_k^\vee + \alpha_k(\alpha_j^\vee)\alpha_j(\alpha_i^\vee) \alpha_k^\vee)
\end{equation}
has a solution for some elements $\phi_{ij}\in X$ satisfying $\phi_{ij} = \phi_{ji}$ for orthogonal roots. Note that $\phi_{ii} = 0$.

First, consider a pair of simple roots $i,j$. From the classification of semisimple groups we have the following cases:
\begin{itemize}
\item $\alpha_j(\alpha_i^\vee) = 0$, i.e. the roots $i$ and $j$ are orthogonal. In this case we let \[\phi_{ij}(\alpha_i^\vee) = \phi_{ji}(\alpha_i^\vee) = 0.\]

\item $\alpha_j(\alpha_i^\vee) = \alpha_i(\alpha_j^\vee) = -1$, i.e. the roots $i$ and $j$ are connected by a single edge. Then we must satisfy the relation
\begin{equation}
d_{iji} = \phi_{ij}(\alpha_i^\vee) + \phi_{ji}(\alpha_j^\vee).
\label{ijitypeA}
\end{equation}

This obviously has a solution for any $d_{iji}$. Note that the relation coming from $d_{jij}$ is the same one.

\item $\alpha_j(\alpha_i^\vee) = -2$ and $\alpha_i(\alpha_j^\vee) = -1$, i.e. the roots $i$ and $j$ are connected by a double edge. We must satisfy the relations
\begin{equation}
\begin{aligned}
d_{iji} &= \phi_{ij}(\alpha_i^\vee) + \phi_{ji}(\alpha_i^\vee) + 2\phi_{ji}(\alpha_j^\vee)\\
d_{jij} &= \phi_{ji}(\alpha_j^\vee) + \phi_{ij}(\alpha_j^\vee) + \phi_{ij}(\alpha_i^\vee).
\end{aligned}
\label{ijitypeB}
\end{equation}

This system has a unique solution for any $\phi_{ji}(\alpha_i^\vee)$ and $\phi_{ij}(\alpha_j^\vee)$:
\begin{align*}
\phi_{ji}(\alpha_j^\vee) &= d_{iji} - d_{jij} + \phi_{ij}(\alpha_j^\vee) - \phi_{ji}(\alpha_i^\vee) \\
\phi_{ij}(\alpha_i^\vee) &= 2d_{jij} - d_{iji} - 2\phi_{ij}(\alpha_j^\vee) + \phi_{ji}(\alpha_i^\vee).
\end{align*}

\item $\alpha_j(\alpha_i^\vee) = -3$ and $\alpha_i(\alpha_j^\vee) = -1$, i.e. the roots $i$ and $j$ are connected by a triple edge. We must satisfy the relations
\begin{equation}
\begin{aligned}
d_{iji} &= \phi_{ij}(\alpha_i^\vee) + 2\phi_{ji}(\alpha_i^\vee) + 3\phi_{ji}(\alpha_j^\vee)\\
d_{jij} &= \phi_{ji}(\alpha_j^\vee) + 2\phi_{ij}(\alpha_j^\vee) + \phi_{ij}(\alpha_i^\vee).
\end{aligned}
\label{ijitypeG}
\end{equation}

This system has a unique solution for any $\phi_{ji}(\alpha_i^\vee)$ and $\phi_{ij}(\alpha_j^\vee) $ if $d_{iji} - d_{jij}$ is even and it has no solutions if the difference $d_{iji} - d_{jij}$ is odd. Therefore, the elements with $d_{iji} - d_{jij}$ odd represent a nontrivial cohomology class if $\alpha_j(\alpha_i^\vee)=-3$.
\end{itemize}

Let us now see what happens for three distinct simple roots $i,j,k$. There are four cases depending on the number of pairs of orthogonal roots among $i,j,k$:
\begin{itemize}
\item All three roots are orthogonal. Then we have to solve the equation
\begin{equation}
d_{ijk} = \phi_{ij}(\alpha_k^\vee) + \phi_{ik}(\alpha_j^\vee) + \phi_{jk}(\alpha_i^\vee).
\label{ijk3}
\end{equation}

It clearly has a solution.

\item There are two orthogonal pairs which we assume to be $(j,k)$ and $(i,k)$. Then we have to solve the system
\begin{equation}
\begin{aligned}
d_{ijk} &= \phi_{ij}(\alpha_k^\vee) + \phi_{ik}(\alpha_j^\vee) + \phi_{jk}(\alpha_i^\vee) - \alpha_j(\alpha_i^\vee)\phi_{jk}(\alpha_j^\vee) \\
d_{kji} &= \phi_{jk}(\alpha_i^\vee) + \phi_{ik}(\alpha_j^\vee) + \phi_{ji}(\alpha_k^\vee) - \alpha_i(\alpha_j^\vee)\phi_{ik}(\alpha_i^\vee)
\end{aligned}
\label{ijk2}
\end{equation}
for the four variables $\phi_{ij}(\alpha_k^\vee), \phi_{ik}(\alpha_j^\vee), \phi_{jk}(\alpha_i^\vee), \phi_{ji}(\alpha_k^\vee)$. It again obviously has a solution.

\item There is a single pair of orthogonal roots which we assume to be $(i, k)$. Then we have to solve the system
\begin{equation}
\begin{aligned}
d_{ijk} &= \phi_{ij}(\alpha_k^\vee) + \phi_{ik}(\alpha_j^\vee) + \phi_{jk}(\alpha_i^\vee) - \alpha_k(\alpha_j^\vee)\phi_{ik}(\alpha_k^\vee) - \alpha_j(\alpha_i^\vee)\phi_{jk}(\alpha_j^\vee) + \alpha_k(\alpha_j^\vee)\alpha_j(\alpha_i^\vee)\phi_{jk}(\alpha_k^\vee) \\
d_{ikj} &= \phi_{ik}(\alpha_j^\vee) + \phi_{ij}(\alpha_k^\vee) + \phi_{kj}(\alpha_i^\vee) - \alpha_j(\alpha_k^\vee)\phi_{ij}(\alpha_j^\vee) - \alpha_j(\alpha_i^\vee)\phi_{kj}(\alpha_j^\vee) \\
d_{jki} &= \phi_{jk}(\alpha_i^\vee) + \phi_{ji}(\alpha_k^\vee) + \phi_{ik}(\alpha_j^\vee) - \alpha_k(\alpha_j^\vee)\phi_{ik}(\alpha_k^\vee) - \alpha_i(\alpha_j^\vee)\phi_{ik}(\alpha_i^\vee) \\
d_{kji} &= \phi_{kj}(\alpha_i^\vee) + \phi_{ik}(\alpha_j^\vee) + \phi_{ji}(\alpha_k^\vee) - \alpha_i(\alpha_j^\vee)\phi_{ik}(\alpha_i^\vee) - \alpha_j(\alpha_k^\vee)\phi_{ji}(\alpha_j^\vee) + \alpha_i(\alpha_j^\vee)\alpha_j(\alpha_k^\vee)\phi_{ji}(\alpha_i^\vee)
\end{aligned}
\label{ijk1}
\end{equation}
for the five variables $\phi_{ij}(\alpha_k^\vee), \phi_{ik}(\alpha_j^\vee), \phi_{jk}(\alpha_i^\vee), \phi_{kj}(\alpha_i^\vee), \phi_{ji}(\alpha_k^\vee)$.

This system has a solution iff
\begin{equation}
\begin{aligned}
&d_{ijk} + d_{kji} + \alpha_j(\alpha_i^\vee)\phi_{jk}(\alpha_j^\vee) - \alpha_k(\alpha_j^\vee)\alpha_j(\alpha_i^\vee)\phi_{jk}(\alpha_k^\vee) + \alpha_j(\alpha_k^\vee)\phi_{ji}(\alpha_j^\vee) - \alpha_i(\alpha_j^\vee)\alpha_j(\alpha_k^\vee)\phi_{ji}(\alpha_i^\vee) \\
&= d_{ikj} + d_{jki} + \alpha_j(\alpha_k^\vee)\phi_{ij}(\alpha_j^\vee) + \alpha_j(\alpha_i^\vee)\phi_{kj}(\alpha_j^\vee).
\end{aligned}
\label{ijk1solvability}
\end{equation}

\item None of the roots are orthogonal. This cannot happen as it would imply that the Dynkin diagram of $G$ has a 2-cycle.
\end{itemize}

To construct a solution $\phi_{ij}(-)$, we run the following algorithm. For each three orthogonal roots $i,j,k$ solve \eqref{ijk3} for $\phi_{ij}(\alpha_k^\vee)$.

Decompose the Dynkin diagram of $G$ into connected components. For each component we do the following. First, for each pair of adjacent roots $i,j$ in the diagram and any other orthogonal root $k$ we solve \eqref{ijk2} for $\phi_{ij}(\alpha_k^\vee)$. We set $\phi_{ij}(\alpha_j^\vee) = 0$. Then depending on the component type we define $\phi_{ij}(\alpha_k^\vee)$ in the following way.

\proofstep{Type $\mA_n$ ($n\geq 3$).} Linearly order the simple roots. For each triple of simple roots $i, j = i+1, k = i+2$ we have to solve
\begin{align*}
d_{ijk} + d_{kji} - \phi_{jk}(\alpha_j^\vee) - \phi_{ji}(\alpha_j^\vee) &= d_{ikj} + d_{jki}\\
d_{iji} &= \phi_{ij}(\alpha_i^\vee) + \phi_{ji}(\alpha_j^\vee)\\
d_{jkj} &= \phi_{jk}(\alpha_j^\vee) + \phi_{kj}(\alpha_k^\vee).
\end{align*}

Set $\phi_{21}(\alpha_2^\vee)=0$. Then for each triple $i$, $j=i+1$ and $k=i+2$ of simple roots in this component we can iteratively solve the system as follows. From the first equation one uniquely determines $\phi_{jk}(\alpha_j^\vee)$. From the last two equations one uniquely determines $\phi_{ij}(\alpha_i^\vee)$ and $\phi_{kj}(\alpha_k^\vee)$.

\proofstep{Type $\mB_n$ ($n\geq 3$).} Linearly order the simple roots starting from the shortest root $i = 1$. For the triple of roots $(i=1, j=2, k=3)$ we get the following equations:
\begin{equation}
\begin{aligned}
d_{321} + d_{123} - \phi_{21}(\alpha_2^\vee) - 2\phi_{23}(\alpha_2^\vee) &= d_{312} + d_{213}\\
d_{212} &= \phi_{21}(\alpha_2^\vee) + \phi_{12}(\alpha_1^\vee)\\
d_{121} &= \phi_{12}(\alpha_1^\vee) + 2\phi_{21}(\alpha_2^\vee)\\
d_{323} &= \phi_{32}(\alpha_3^\vee) + \phi_{23}(\alpha_2^\vee).
\end{aligned}
\label{typeBsystem}
\end{equation}

Let us substitute $\phi_{21}(\alpha_2^\vee)$ from the second and third equations into the first one. Then we obtain
\[d_{321} + d_{123} - d_{121} + d_{212} - 2\phi_{23}(\alpha_2^\vee) = d_{312} + d_{213}.\]

This equation has a solution iff $d_{321} + d_{123} - d_{121} + d_{212} - d_{312} - d_{213}$ is even. Moreover, the whole system \eqref{typeBsystem} has a solution in this case. Note that even if we did not set $\phi_{ij}(\alpha_j^\vee) = 0$ we would still not be able to solve the equation unless the same condition is satisfied. For the other triple of adjacent roots we have to solve the same equations as in the type $\mA$ case.

\proofstep{Type $\mC_n$ ($n\geq 3$).} Linearly order the simple roots starting from the \textit{longest} root $i = 1$. Then for the triple of roots $(i = 1, j = 2, k = 3)$ we get the following equations:
\begin{equation}
\begin{aligned}
d_{123} + d_{321} - \phi_{23}(\alpha_2^\vee) - \phi_{21}(\alpha_2^\vee) &= d_{132} + d_{231}\\
d_{121} &= \phi_{12}(\alpha_1^\vee) + \phi_{21}(\alpha_2^\vee)\\
d_{212} &= \phi_{21}(\alpha_2^\vee) + 2\phi_{12}(\alpha_1^\vee) \\
d_{232} &= \phi_{23}(\alpha_2^\vee) + \phi_{32}(\alpha_3^\vee).
\end{aligned}
\label{typeCsystem}
\end{equation}

Substituting $\phi_{21}(\alpha_2^\vee)$ from the middle two equations into the first one, we obtain
\[d_{123} + d_{321} - \phi_{23}(\alpha_2^\vee) - 2d_{121} + d_{212} = d_{132} + d_{231}.\]

From this equation we determine $\phi_{23}(\alpha_2^\vee)$ and we uniquely determine $\phi_{32}(\alpha_3^\vee)$, $\phi_{12}(\alpha_1^\vee)$ and $\phi_{21}(\alpha_2^\vee)$ from the rest of the equations in the system \eqref{typeCsystem}. We proceed with the rest of the diagram as in type $\mA$.

\proofstep{Type $\mD_n$ ($n\geq 4$) or $\mE_n$.} We order the diagram as in the picture:
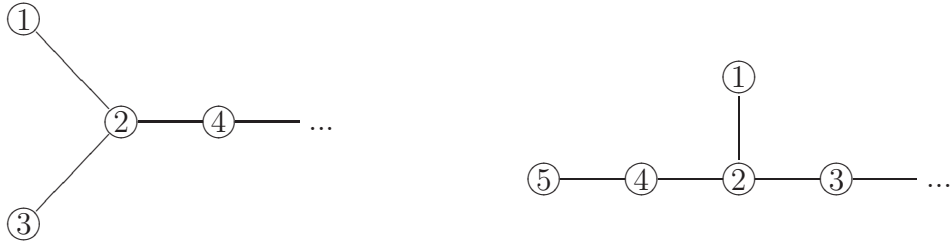
\begin{center}
\begin{figure}[h]
\begin{minipage}{0.45\textwidth}
\[
\xymatrix{
*+[o][F]{1} \ar@{-}[dr] & & & \\
& *+[o][F]{2} \ar@{-}[r] & *+[o][F]{4} \ar@{-}[r] & ... \\
*+[o][F]{3} \ar@{-}[ur] & & &}
\]
\end{minipage}
\begin{minipage}{0.45\textwidth}
\[
\xymatrix{
& & *+[o][F]{1} \ar@{-}[d] & & \\
*+[o][F]{5} \ar@{-}[r] & *+[o][F]{4} \ar@{-}[r] & *+[o][F]{2} \ar@{-}[r] & *+[o][F]{3} \ar@{-}[r] & ...
}
\]
\end{minipage}
\end{figure}
\end{center}

The first four roots give the equations

\begin{align*}
d_{123} + d_{321} - \phi_{23}(\alpha_2^\vee) - \phi_{23}(\alpha_3^\vee) - \phi_{21}(\alpha_2^\vee) - \phi_{21}(\alpha_1^\vee) &= d_{132} + d_{231} - \phi_{12}(\alpha_2^\vee) - \phi_{32}(\alpha_2^\vee) \\
d_{124} + d_{421} - \phi_{24}(\alpha_2^\vee) - \phi_{24}(\alpha_4^\vee) - \phi_{21}(\alpha_2^\vee) - \phi_{21}(\alpha_1^\vee) &= d_{142} + d_{241} - \phi_{12}(\alpha_2^\vee) - \phi_{42}(\alpha_2^\vee) \\
d_{324} + d_{423} - \phi_{24}(\alpha_2^\vee) - \phi_{24}(\alpha_4^\vee) - \phi_{23}(\alpha_2^\vee) - \phi_{23}(\alpha_3^\vee) &= d_{342} + d_{243} - \phi_{32}(\alpha_2^\vee) - \phi_{42}(\alpha_2^\vee).
\end{align*}

If we add the three equations, we see that the system has solutions only if
\[d_{123} + d_{321} - d_{132} - d_{231} + d_{124} + d_{421} - d_{142} - d_{241} + d_{324} + d_{423} - d_{342} - d_{243}\]
is even.

Let us now show that if this expression is even, we indeed have a solution. Set
\[\phi_{12}(\alpha_2^\vee) = \phi_{21}(\alpha_1^\vee) = \phi_{23}(\alpha_3^\vee) = \phi_{32}(\alpha_2^\vee) = \phi_{24}(\alpha_4^\vee) = \phi_{42}(\alpha_2^\vee) = 0.\]

Then the equations become
\begin{equation}
\begin{aligned}
d_{123} + d_{321} - \phi_{23}(\alpha_2^\vee) - \phi_{21}(\alpha_2^\vee) &= d_{132} + d_{231} \\
d_{124} + d_{421} - \phi_{24}(\alpha_2^\vee) - \phi_{21}(\alpha_2^\vee) &= d_{142} + d_{241} \\
d_{324} + d_{423} - \phi_{24}(\alpha_2^\vee) - \phi_{23}(\alpha_2^\vee) &= d_{342} + d_{243}.
\end{aligned}
\label{typeDsystem}
\end{equation}

From the last equation we have
\[\phi_{24}(\alpha_2^\vee) = d_{324} + d_{423} - d_{342} - d_{243} - \phi_{23}(\alpha_2^\vee).\]
Plug it into the second equation:
\[d_{124} + d_{421} - d_{324} - d_{423} + d_{342} + d_{243} + \phi_{23}(\alpha_2^\vee) - \phi_{21}(\alpha_2^\vee) = d_{142} + d_{241}.\]

This and the first equation in \eqref{typeDsystem} have a unique solution under the assumption that the expression written before is even. The tails of the diagram are treated in the same way as in type $\mA$.

\proofstep{Type $\mF_4$.} We order the four simple roots in the following way:
\begin{center}
\begin{figure}[h]
\[
\xymatrix{
*+[o][F]{1} \ar@{-}[r] & *+[o][F]{2} \ar@{=}[r]|*{>} & *+[o][F]{3} \ar@{-}[r] & *+[o][F]{4}
}
\]
\end{figure}
\end{center}

As in type $\mB$, we can solve the system \eqref{typeBsystem} iff $d_{321} + d_{123} - d_{121} + d_{212} - d_{312} - d_{213}$ is even. We get a system of equations identical to \eqref{typeCsystem} for the roots $234$, which always has a solution.

\proofstep{Diagram has rank 2.} We determine $\phi_{ij}(\alpha_i^\vee)$ and $\phi_{ji}(\alpha_j^\vee)$ from the equations \eqref{ijitypeA}, \eqref{ijitypeB}, \eqref{ijitypeG}.

In this way we have determined $\phi_{ij}(\alpha_k^\vee)$ for any three simple roots unless there is a subdiagram of type $\mB_3$ (in types $\mB$ and $\mF$), $\mD_4$ (in types $\mD$ and $\mE$) or $\mG_2$ in which case an extra solvability condition is required.
\end{proof}

\begin{remark}
The fact that $\CH^3(G)$ is 2-torsion follows from the following argument (\cite[Proof of Proposition 3.20]{EKLV}). Levine \cite[Theorem 2.1]{Lev} shows that under our assumptions $\K_0(G)\cong \Z$ with the trivial coniveau (or topological) filtration. There is a surjective map \[\pi_3\colon \CH^3(G)\rightarrow \gr^3 \K_0(G)\cong 0\] (see \cite[Exp. 0, Th\'{e}or\`{e}mes 2.6, 2.12]{SGA6}) given by sending a cycle to the skyscraper sheaf. There is the third Chern class map $c_3$ in the other direction
\[c_3\colon \gr_3 \K_0(G)\rightarrow \CH^3(G).\]
By the Grothendieck-Riemann-Roch theorem without denominators we have $c_3\pi_3 = 2\id$. In particular, the kernel of $\pi_3$, which is the whole $\CH^3(G)$ in our case, is annihilated by $2$.
\end{remark}

We now proceed to the computation of the cohomology $E_1^{-3, 5}$.

Given a cubic form $C(-)$ on a lattice $Y$, we have the associated quadratic-linear form
\[B(y_1, y_2) = \frac{C(y_2 + y_1) + C(y_2 - y_1)}{2} - C(y_2),\]
which is quadratic in the first argument and linear in the second, and the associated trilinear form
\[T(y_1, y_2, y_3) = C(y_1 + y_2 + y_3) - C(y_1 + y_2) - C(y_2 + y_3) - C(y_1 + y_3) + C(y_1) + C(y_2) + C(y_3).\]
Note that $B$ is an integer-valued form despite the factor of $1/2$.

We have the following relations between $B$ and $T$:
\begin{equation}
T(y_1, y_1, y_2) = 2B(y_1, y_2),\quad B(y_1 + y_2, y_3) = B(y_1, y_3) + B(y_2, y_3) + T(y_1, y_2, y_3).
\label{BTrelation}
\end{equation}

If $\{\alpha_i^\vee\}_i$ is a basis of $Y$, the cubic form $C$ is uniquely determined by the values $C(\alpha_i^\vee)$, $B(\alpha_i^\vee, \alpha_j^\vee)$ for $i\neq j$ and $T(\alpha_i^\vee, \alpha_j^\vee, \alpha_k^\vee)$ for $i,j,k$ distinct.

\begin{prop}
\label{prop:Winvcubic}
A cubic form $C$ on the cocharacter lattice $Y$ is $W$-invariant iff
\begin{enumerate}
\item $C(\alpha_i^\vee) = 0$ for all simple coroots $\alpha_i^\vee$

\item $B(\alpha_j^\vee, \alpha_i^\vee) = \alpha_i(\alpha_j^\vee) B(\alpha_i^\vee, \alpha_j^\vee)$ for distinct simple coroots $\alpha_i^\vee$ and $\alpha_j^\vee$ and

\item
\begin{align*}
0 &= T(\alpha_i^\vee, \alpha_j^\vee, \alpha_k^\vee),\\
0 &= \alpha_i(\alpha_j^\vee)B(\alpha_i^\vee, \alpha_k^\vee) + \alpha_i(\alpha_k^\vee)B(\alpha_i^\vee, \alpha_j^\vee)
\end{align*} for distinct simple coroots $\alpha_i^\vee, \alpha_j^\vee, \alpha_k^\vee$.
\end{enumerate}
\end{prop}
\begin{proof}
A cubic form $C$ is $W$-invariant iff $C(s_i(y)) = C(y)$ for every $y\in Y$ and every simple reflection $s_i$. For instance, for $y = \alpha_i^\vee$ we get $C(-\alpha_i^\vee) = C(\alpha_i^\vee)$ which implies that $C(\alpha_i^\vee) = 0$.

Expanding the $W$-invariance condition, we get
\begin{align*}
C(y) &= C(s_i(y))\\
&= C(y - \alpha_i(y)\alpha_i^\vee)\\
&= C(y) - C(\alpha_i(y)\alpha_i^\vee) + B(y, -\alpha_i(y)\alpha_i^\vee) + B(-\alpha_i(y)\alpha_i^\vee, y).
\end{align*}

Since $\alpha_i$ is a non-zero functional, $W$-invariance is equivalent to $C(\alpha_i^\vee) = 0$ and
\[0 = -B(y, \alpha_i^\vee) + \alpha_i(y) B(\alpha_i^\vee, y).\]

This expression is quadratic in $y$, so it is enough to check it on $y = a\alpha_j^\vee + b\alpha_k^\vee$ for every integers $a$ and $b$. Then we get
\begin{align*}
0 =& -B(a\alpha_j^\vee + b\alpha_k^\vee, \alpha_i^\vee) + \alpha_i(y) B(\alpha_i^\vee, a\alpha_j^\vee + b\alpha_k^\vee) \\
=&-B(a\alpha_j^\vee, \alpha_i^\vee) - B(b\alpha_k^\vee, \alpha_i^\vee)-abT(\alpha_j^\vee, \alpha_k^\vee, \alpha_i^\vee) \\
&+ (a\alpha_i(\alpha_j^\vee) + b\alpha_i(\alpha_k^\vee))(aB(\alpha_i^\vee, \alpha_j^\vee) + bB(\alpha_i^\vee, \alpha_k^\vee)).
\end{align*}
Here we have used the relation \eqref{BTrelation} between $B$ and $T$.

This equation has to hold identically in $a$ and $b$. Therefore, $W$-invariance is equivalent to the following three equations:
\begin{align*}
0 &= C(\alpha_i^\vee) \\
0 &= -B(\alpha_j^\vee, \alpha_i^\vee) + \alpha_i(\alpha_j^\vee)B(\alpha_i^\vee, \alpha_j^\vee) \\
T(\alpha_i^\vee, \alpha_j^\vee, \alpha_k^\vee) &= \alpha_i(\alpha_j^\vee)B(\alpha_i^\vee, \alpha_k^\vee) + \alpha_i(\alpha_k^\vee) B(\alpha_i^\vee, \alpha_j^\vee).
\end{align*}

From the second equation we see that $B(\alpha_j^\vee, \alpha_i^\vee) = 0$ for orthogonal coroots $\alpha_i^\vee, \alpha_j^\vee$. Since in any triple of distinct simple coroots two are orthogonal, without loss of generality we will assume that $i$ and $k$ are orthogonal. Then from the last equation we get that $T(\alpha_i^\vee, \alpha_j^\vee, \alpha_k^\vee) = 0$ for any triple of distinct simple coroots.
\end{proof}

We have the following standard fact about cubic invariants which follows from examining the list of fundamental invariants of Weyl groups.

\begin{lm}
The abelian group of $W\!$-invariant cubic forms on $Y$ is free of rank equal to the number of type $\mA$ factors in $G$.
\label{lm:Winvcubicclass}
\end{lm}

Our goal will be to show that $E_1^{-3, 5}$ is isomorphic to the group of $W$-invariant cubic forms on $Y$. We begin by defining a map from $E_0^{-3, 5}$ to cubic forms on $Y$.

Given a collection $\{\phi_{ij}\}_{i,j}\in E_0^{-3, 5}$ of elements $\phi_{ij}\in X$, we define a cubic form by
\[C(\alpha_i^\vee) = 0,\quad B(\alpha_i^\vee, \alpha_j^\vee) = \phi_{ij}(\alpha_i^\vee) - \phi_{ji}(\alpha_i^\vee) - \alpha_j(\alpha_i^\vee)\phi_{ij}(\alpha_j^\vee),\quad T(\alpha_i^\vee, \alpha_j^\vee, \alpha_k^\vee) = 0.\]

\begin{prop}
This map descends to an isomorphism from $E_1^{-3, 5}$ to the group of $W\!\text{-invariant}$ cubic forms on $Y$.
\end{prop}
\begin{proof}
We split the proof of this proposition into four steps.

\proofstep{Step 1.} The map annihilates exact elements of $E_0^{-3, 5}$.

The collection $\{\phi_{ij}\}_{i,j}\in E_0^{-3, 5}$ is exact if there is a collection $\{D_i\}_{i}\in E_0^{-3, 4}$ of elements $D_i\in \wedge^2 X$, such that
\[\phi_{ij}(\alpha_k^\vee) = D_i(\alpha_j^\vee, \alpha_k^\vee) + D_j(\alpha_i^\vee, \alpha_k^\vee) - \alpha_j(\alpha_i^\vee)D_j(\alpha_j^\vee, \alpha_k^\vee).\]

Then
\begin{align*}
\phi_{ij}(\alpha_i^\vee) &= D_i(\alpha_j^\vee, \alpha_i^\vee) - \alpha_j(\alpha_i^\vee)D_j(\alpha_j^\vee, \alpha_i^\vee) \\
\phi_{ji}(\alpha_i^\vee) &= D_i(\alpha_j^\vee, \alpha_i^\vee)\\
\phi_{ij}(\alpha_j^\vee) &= D_j(\alpha_i^\vee, \alpha_j^\vee).
\end{align*}

Therefore,
\begin{align*}
B(\alpha_i^\vee, \alpha_j^\vee) &= \phi_{ij}(\alpha_i^\vee) - \phi_{ji}(\alpha_i^\vee) - \alpha_j(\alpha_i^\vee)\phi_{ij}(\alpha_j^\vee) \\
&= D_i(\alpha_j^\vee, \alpha_i^\vee) - \alpha_j(\alpha_i^\vee) D_j(\alpha_j^\vee, \alpha_i^\vee) - D_i(\alpha_j^\vee, \alpha_i^\vee) - \alpha_j(\alpha_i^\vee)D_j(\alpha_i^\vee, \alpha_j^\vee)\\
&= 0.
\end{align*}

In the last line we have used antisymmetry of $D_j$.

\proofstep{Step 2.} The map sends closed elements of $E_0^{-3, 5}$ to $W$-invariant cubic forms.

The collection $\{\phi_{ij}\}_{i,j}\in E_0^{-3, 5}$ is closed if
\begin{align*}
0 = &\phi_{ij}(\alpha_k^\vee) + \phi_{ik}(\alpha_j^\vee) + \phi_{jk}(\alpha_i^\vee) - \alpha_k(\alpha_j^\vee) \phi_{ik}(\alpha_k^\vee) - \alpha_j(\alpha_i^\vee)\phi_{jk}(\alpha_j^\vee)\\
&+ (\alpha_k(\alpha_j^\vee)\alpha_j(\alpha_i^\vee) - \alpha_k(\alpha_i^\vee))\phi_{jk}(\alpha_k^\vee).
\end{align*}

We have to show that such $\phi_{ij}$ are sent to cubic forms satisfying
\begin{align*}
B(\alpha_j^\vee, \alpha_i^\vee) &= \alpha_i(\alpha_j^\vee) B(\alpha_i^\vee, \alpha_j^\vee) \\
0 &= \alpha_i(\alpha_j^\vee) B(\alpha_i^\vee, \alpha_k^\vee) + \alpha_i(\alpha_k^\vee) B(\alpha_i^\vee, \alpha_j^\vee).
\end{align*}

In other words, we have to show that $\phi_{ij}$ satisfy
\begin{equation}
\begin{aligned}
\phi_{ji}(\alpha_j^\vee) - \phi_{ij}(\alpha_j^\vee) - \alpha_i(\alpha_j^\vee)\phi_{ji}(\alpha_i^\vee) &= \alpha_i(\alpha_j^\vee)(\phi_{ij}(\alpha_i^\vee) - \phi_{ji}(\alpha_i^\vee) - \alpha_j(\alpha_i^\vee)\phi_{ij}(\alpha_j^\vee)) \\
\alpha_i(\alpha_k^\vee)(\phi_{ij}(\alpha_i^\vee) - \phi_{ji}(\alpha_i^\vee) - \alpha_j(\alpha_i^\vee)\phi_{ij}(\alpha_j^\vee)) &= -\alpha_i(\alpha_j^\vee)(\phi_{ik}(\alpha_i^\vee) - \phi_{ki}(\alpha_i^\vee) - \alpha_k(\alpha_i^\vee)\phi_{ik}(\alpha_k^\vee)).
\end{aligned}
\label{phiijWinv}
\end{equation}

The last equation is automatically satisfied if $\alpha_i^\vee$ is orthogonal to either $\alpha_j^\vee$ or $\alpha_k^\vee$. Since among three distinct roots two are necessarily orthogonal, we just need to consider the case when $\alpha_j^\vee$ is orthogonal to $\alpha_k^\vee$. Consider permutations of the indices $(i,j,k)$ in the closedness equation:
\begin{equation}
\begin{aligned}
0 &= \phi_{ij}(\alpha_k^\vee) + \phi_{ik}(\alpha_j^\vee) + \phi_{jk}(\alpha_i^\vee) - \alpha_j(\alpha_i^\vee)\phi_{jk}(\alpha_j^\vee)- \alpha_k(\alpha_i^\vee)\phi_{jk}(\alpha_k^\vee) \\
0 &= \phi_{ki}(\alpha_j^\vee) + \phi_{jk}(\alpha_i^\vee) + \phi_{ij}(\alpha_k^\vee) - \alpha_j(\alpha_i^\vee) \phi_{jk}(\alpha_j^\vee) - \alpha_i(\alpha_k^\vee)\phi_{ij}(\alpha_i^\vee) + \alpha_j(\alpha_i^\vee)\alpha_i(\alpha_k^\vee)\phi_{ij}(\alpha_j^\vee) \\
0 &= \phi_{ji}(\alpha_k^\vee) + \phi_{jk}(\alpha_i^\vee) + \phi_{ik}(\alpha_j^\vee) - \alpha_k(\alpha_i^\vee) \phi_{jk}(\alpha_k^\vee) - \alpha_i(\alpha_j^\vee)\phi_{ik}(\alpha_i^\vee) + \alpha_k(\alpha_i^\vee)\alpha_i(\alpha_j^\vee)\phi_{ik}(\alpha_k^\vee) \\
0 &= \phi_{jk}(\alpha_i^\vee) + \phi_{ji}(\alpha_k^\vee) + \phi_{ki}(\alpha_j^\vee) - \alpha_i(\alpha_k^\vee) \phi_{ji}(\alpha_i^\vee) - \alpha_i(\alpha_j^\vee)\phi_{ki}(\alpha_i^\vee).
\end{aligned}
\label{ijkclosed}
\end{equation}

Taking the sum of the second and third equations and subtracting the first and fourth equations we arrive precisely at the second equation in \eqref{phiijWinv}.

The closedness conditions for $k = i$ are
\begin{align*}
0 &= \phi_{ij}(\alpha_i^\vee) - \alpha_j(\alpha_i^\vee)\phi_{ji}(\alpha_j^\vee) + (\alpha_i(\alpha_j^\vee)\alpha_j(\alpha_i^\vee) - 1)\phi_{ji}(\alpha_i^\vee) \\
0 &= \phi_{ji}(\alpha_j^\vee) - \alpha_i(\alpha_j^\vee)\phi_{ij}(\alpha_i^\vee) + (\alpha_i(\alpha_j^\vee)\alpha_j(\alpha_i^\vee) - 1)\phi_{ij}(\alpha_j^\vee).
\end{align*}

The first equation coincides with the first equation in \eqref{phiijWinv}.

\proofstep{Step 3.} The map from $E_1^{-3, 5}$ to $W$-invariant cubic forms is injective.

Suppose that $B(\alpha_i^\vee, \alpha_j^\vee) = 0$, i.e.
\[\phi_{ij}(\alpha_i^\vee) = \phi_{ji}(\alpha_i^\vee) + \alpha_j(\alpha_i^\vee)\phi_{ij}(\alpha_j^\vee)\]
for every pair of simple coroots $\alpha_i^\vee, \alpha_j^\vee$. Then we need to prove that the corresponding elements $\phi_{ij}$ are exact, i.e. there are $D_i\in\wedge^2 X$, such that
\[\phi_{ij}(\alpha_k^\vee) = D_i(\alpha_j^\vee, \alpha_k^\vee) + D_j(\alpha_i^\vee, \alpha_k^\vee) - \alpha_j(\alpha_i^\vee)D_j(\alpha_j^\vee, \alpha_k^\vee).\]

First, consider the case of two distinct coroots $i,j$. Then we get the equations
\begin{align*}
\phi_{ij}(\alpha_j^\vee) &= D_j(\alpha_i^\vee, \alpha_j^\vee) \\
\phi_{ij}(\alpha_i^\vee) &= D_i(\alpha_j^\vee, \alpha_i^\vee) + \alpha_j(\alpha_i^\vee) D_j(\alpha_i^\vee, \alpha_j^\vee)\\
\phi_{ji}(\alpha_i^\vee) &= D_i(\alpha_j^\vee, \alpha_i^\vee) \\
\phi_{ji}(\alpha_j^\vee) &= D_j(\alpha_i^\vee, \alpha_j^\vee) + \alpha_i(\alpha_j^\vee) D_i(\alpha_j^\vee, \alpha_i^\vee).
\end{align*}

From the first and third equations we determine $D_j(\alpha_i^\vee, \alpha_j^\vee)$ and $D_i(\alpha_j^\vee, \alpha_i^\vee)$. The second and fourth equations will be satisfied if
\[\phi_{ij}(\alpha_i^\vee) = \phi_{ji}(\alpha_i^\vee) + \alpha_j(\alpha_i^\vee)\phi_{ij}(\alpha_j^\vee)\]
and similarly after interchanging $i$ and $j$. But this equation is exactly $B(\alpha_i^\vee, \alpha_j^\vee) = 0$.

Now consider the case of three distinct coroots $(i, j, k)$, where we again assume that $\alpha_j^\vee$ and $\alpha_k^\vee$ are orthogonal. The equations we get are
\begin{align*}
\phi_{ij}(\alpha_k^\vee) &= D_i(\alpha_j^\vee, \alpha_k^\vee) + D_j(\alpha_i^\vee, \alpha_k^\vee) + \alpha_j(\alpha_i^\vee)D_j(\alpha_k^\vee, \alpha_j^\vee) \\
\phi_{ji}(\alpha_k^\vee) &= D_i(\alpha_j^\vee, \alpha_k^\vee) + D_j(\alpha_i^\vee, \alpha_k^\vee) + \alpha_i(\alpha_j^\vee)D_i(\alpha_k^\vee, \alpha_i^\vee) \\
\phi_{jk}(\alpha_i^\vee) &= -D_k(\alpha_i^\vee, \alpha_j^\vee) - D_j(\alpha_i^\vee, \alpha_k^\vee) \\
\phi_{ik}(\alpha_j^\vee) &= -D_i(\alpha_j^\vee, \alpha_k^\vee) + D_k(\alpha_i^\vee, \alpha_ j^\vee) + \alpha_k(\alpha_i^\vee)D_k(\alpha_j^\vee, \alpha_k^\vee) \\
\phi_{ki}(\alpha_j^\vee) &= -D_i(\alpha_j^\vee, \alpha_k^\vee) + D_k(\alpha_i^\vee, \alpha_j^\vee) + \alpha_i(\alpha_k^\vee)D_i(\alpha_j^\vee, \alpha_i^\vee).
\end{align*}

Let us substitute $D_j(\alpha_k^\vee, \alpha_j^\vee)$ and its permutations in terms of $\phi_{kj}(\alpha_j^\vee)$ and its permutations:
\begin{align*}
\phi_{ij}(\alpha_k^\vee) &= D_i(\alpha_j^\vee, \alpha_k^\vee) + D_j(\alpha_i^\vee, \alpha_k^\vee) + \alpha_j(\alpha_i^\vee)\phi_{kj}(\alpha_j^\vee) \\
\phi_{ji}(\alpha_k^\vee) &= D_i(\alpha_j^\vee, \alpha_k^\vee) + D_j(\alpha_i^\vee, \alpha_k^\vee) + \alpha_i(\alpha_j^\vee)\phi_{ki}(\alpha_i^\vee) \\
\phi_{jk}(\alpha_i^\vee) &= -D_k(\alpha_i^\vee, \alpha_j^\vee) - D_j(\alpha_i^\vee, \alpha_k^\vee) \\
\phi_{ik}(\alpha_j^\vee) &= -D_i(\alpha_j^\vee, \alpha_k^\vee) + D_k(\alpha_i^\vee, \alpha_ j^\vee) + \alpha_k(\alpha_i^\vee)\phi_{jk}(\alpha_k^\vee) \\
\phi_{ki}(\alpha_j^\vee) &= -D_i(\alpha_j^\vee, \alpha_k^\vee) + D_k(\alpha_i^\vee, \alpha_j^\vee) + \alpha_i(\alpha_k^\vee)\phi_{ji}(\alpha_i^\vee).
\end{align*}

The solvability equation for the first two equations is
\[\phi_{ij}(\alpha_k^\vee) - \alpha_j(\alpha_i^\vee)\phi_{kj}(\alpha_j^\vee) = \phi_{ji}(\alpha_k^\vee) - \alpha_i(\alpha_j^\vee)\phi_{ki}(\alpha_i^\vee).\]

Let's show that it follows from the closedness equations \eqref{ijkclosed}. Indeed, equating the second and fourth equations, we get
\begin{align*}
&\phi_{ij}(\alpha_k^\vee) - \alpha_j(\alpha_i^\vee)\phi_{jk}(\alpha_j^\vee) - \alpha_i(\alpha_k^\vee)\phi_{ij}(\alpha_i^\vee) + \alpha_j(\alpha_i^\vee)\alpha_i(\alpha_k^\vee)\phi_{ij}(\alpha_j^\vee) \\
&= \phi_{ji}(\alpha_k^\vee) - \alpha_i(\alpha_k^\vee)\phi_{ji}(\alpha_i^\vee) - \alpha_i(\alpha_j^\vee)\phi_{ki}(\alpha_i^\vee).
\end{align*}

The terms
\[-\alpha_i(\alpha_k^\vee)\phi_{ij}(\alpha_i^\vee) + \alpha_j(\alpha_i^\vee)\alpha_i(\alpha_k^\vee)\phi_{ij}(\alpha_j^\vee) + \alpha_i(\alpha_k^\vee)\phi_{ji}(\alpha_i^\vee)\]
add up to $-\alpha_i(\alpha_k^\vee) B(\alpha_i^\vee, \alpha_j^\vee) = 0$. The rest coincides with the solvability equation.

Similarly, the solvability equation for the last two equations is satisfied if we repeat the argument with $j$ replaced by $k$.

Therefore, we just need to be able to solve a smaller system
\begin{align*}
\phi_{ij}(\alpha_k^\vee) &= D_i(\alpha_j^\vee, \alpha_k^\vee) + D_j(\alpha_i^\vee, \alpha_k^\vee) + \alpha_j(\alpha_i^\vee)\phi_{kj}(\alpha_j^\vee) \\
\phi_{jk}(\alpha_i^\vee) &= -D_k(\alpha_i^\vee, \alpha_j^\vee) - D_j(\alpha_i^\vee, \alpha_k^\vee) \\
\phi_{ik}(\alpha_j^\vee) &= -D_i(\alpha_j^\vee, \alpha_k^\vee) + D_k(\alpha_i^\vee, \alpha_ j^\vee) + \alpha_k(\alpha_i^\vee)\phi_{jk}(\alpha_k^\vee).
\end{align*}

This system of equations on $D_i(\alpha_j^\vee, \alpha_k^\vee)$ and its permutations has a (non-unique) solution if the following solvability equation is satisfied:
\[\phi_{ij}(\alpha_k^\vee) + \phi_{jk}(\alpha_i^\vee) + \phi_{ik}(\alpha_j^\vee) = \alpha_j(\alpha_i^\vee)\phi_{kj}(\alpha_j^\vee) + \alpha_k(\alpha_i^\vee)\phi_{jk}(\alpha_k^\vee).\]

This solvability equation is equivalent to the first equation in \eqref{ijkclosed}. Therefore, we can find the antisymmetric forms $D_i$ and hence $\{\phi_{ij}\}_{i,j}$ is exact.

\proofstep{Step 4.} The map from $E_1^{-3, 5}$ to $W$-invariant cubic forms is surjective.

Without loss of generality, we will assume that $G$ is of type $A$ (see Lemma \ref{lm:Winvcubicclass}). Linearly order the roots and let
\[\phi_{i,i+1}(\alpha_i^\vee) = -\phi_{i+1,i}(\alpha_{i+1}^\vee) = B(\alpha_i^\vee, \alpha_{i+1}^\vee),\quad \phi_{ij}(i) = -\phi_{ij}(j).\]

This is clearly a lift of the cubic form given by $B$ and we just have to show that thus defined $\{\phi_{ij}\}_{i,j}$ is closed. Indeed, we need to check that
\[\phi_{ij}(\alpha_i^\vee) + \phi_{ji}(\alpha_j^\vee) = 0,\]
which is clearly satisfied by assumption, and the closedness equations \eqref{ijkclosed}.

The closedness equations are nontrivial only for a triple of consecutive roots $j$, $i=j+1$ and $k=j+2$:
\begin{align*}
0 &= 0 \\
0 &= \phi_{ij}(\alpha_i^\vee) + \phi_{ij}(\alpha_j^\vee) \\
0 &= \phi_{ik}(\alpha_i^\vee) + \phi_{ik}(\alpha_k^\vee) \\
0 &= \phi_{ji}(\alpha_i^\vee) + \phi_{ki}(\alpha_i^\vee).
\end{align*}

The first three equations are trivially satisfied from our assumptions and the last one follows from the fact that cubic invariants have $B(\alpha_j^\vee, \alpha_{j+1}^\vee) = B(\alpha_{j+1}^\vee, \alpha_{j+2}^\vee)$. \qedhere
\end{proof}

The $E_1$ page of the spectral sequence is thus
\[
\xymatrix@R-2pc@C-2pc{
E_1^{-3,6} & 0 & 0 & 0 \\
E_1^{-3, 5} & 0 & 0 & 0 \\
0 & 0 & 0 & 0\\
0 & E_1^{-2, 3} & 0 & 0 \\
0 & 0 & 0 & 0 \\
0 & 0 & 0 & 0 \\
0 & 0 & 0 & \K_3(k)
}\]

Therefore, the spectral sequence again degenerates at $E_1$ and we derive the desired cohomology groups.

\section{$\cK$-cohomology of reductive groups}
\label{sect:kcohomologyreductive}

In this section $G$ will denote a split connected reductive group over a field $k$ with a simply-connected derived group $G_{der}$. Thus, we have an exact sequence
\[1\rightarrow G_{der}\rightarrow G\rightarrow H_0\rightarrow 1\]
which induces an exact sequence of character lattices
\[0\rightarrow X_0\rightarrow X\rightarrow X_{der}\rightarrow 0\]
and cocharacter lattices
\[0\rightarrow Y_{der}\rightarrow Y\rightarrow Y_0\rightarrow 0.\]

Our goal will be to compute the $\cK_3$-cohomology of $G$.

\subsection{$\cK_2$-cohomology}

As a warm-up, let's compute the $\cK_2$-cohomology of $G$. Recall the $E_0$ page of the Brylinski--Deligne spectral sequence computing $\mH^\bullet(G, \cK_2)$:
\[
\xymatrix@R-1pc@C-2pc{
\oplus_{W^{(2)}} \Z & 0 & 0 \\
\oplus_{W^{(1)}} X \ar[u] & 0 & 0 \\
\bigwedge^2 X \ar[u] & \oplus_{W^{(1)}} k^\times & 0 \\
0 & X\otimes k^\times \ar[u] & 0 \\
0 & 0 & \K_2(k)
}
\]

By simply-connectivity of $G_{der}$ we can identify $\oplus_{W^{(1)}}\Z\cong X_{der}$. Therefore, the cohomology in the $(-1)$-st column is
\[E_1^{-1, 1} = X_0\otimes k^\times,\quad E_1^{-1, 2} = 0.\]

The differential
\[\d\colon \bigwedge^2 X\rightarrow X_{der}\otimes X\]
is no longer injective, but its kernel is $E_1^{-2, 2} = \wedge^2 X_0$.

The differential
\[\d\colon \oplus_{W^{(1)}} X\rightarrow \oplus_{W^{(2)}}\Z\]
is still surjective, so $E_1^{-2, 4} = 0$.

We compute the cohomology $E_1^{-2, 3}$ in the following way. First, it coincides with the cohomology of
\[\bigwedge^2 X / \bigwedge^2 X_0\rightarrow X_{der}\otimes X\rightarrow \oplus_{W^{(2)}}\Z.\]

In other words, we have to compute cohomology of the middle column in
\[
\xymatrix{
0 \ar[r] & X_{der}\otimes X_0 \ar[r] \ar^{1}[d] & \wedge^2 X / \wedge^2 X_0 \ar[r] \ar[d] & \wedge^2 X_{der} \ar[r] \ar[d] & 0 \\
0 \ar[r] & X_{der}\otimes X_0 \ar[r] \ar[d] & X_{der}\otimes X \ar[r] \ar[d] & X_{der}\otimes X_{der} \ar[r] \ar[dl] & 0 \\
& 0 & \oplus_{W^{(2)}}\Z & &
}
\]

Since the vertical map in the left-most column is an isomorphism, the cohomology in the middle column coincides with the cohomology in the right-most column, which we have already computed. Thus, $E_1^{-2, 3}$ coincides with the group of $W$-invariant quadratic forms on $Y_{der}$.

The $E_1$ page of the spectral sequence is
\[
\xymatrix@R-2pc@C-2pc{
E_1^{-2, 3} & 0 & 0 \\
\bigwedge^2 X_0 & 0 & 0 \\
0 & X_0\otimes k^\times & 0 \\
0 & 0 & \K_2(k)
}
\]

The spectral sequence degenerates at $E_1$. Denote by $E_n(H_0)$ the same spectral sequence for the group $H_0$. As we already know, the spectral sequence is concentrated on the diagonal $p+q = 0$ and, moreover, the pullback map $\mH^0(H_0, \cK_2)\rightarrow \mH^0(G, \cK_2)$ induces isomorphisms on the diagonal terms starting at the $E_1$ page.

\begin{thm}[Brylinski--Deligne]
Let $G$ be a split connected reductive group with a simply-connected derived group. Then the $\cK_2$-cohomology groups are as follows:
\begin{itemize}
\item $\mH^0(G, \cK_2)\cong \mH^0(H_0, \cK_2)$ is isomorphic to the group of alternating forms $A$ on the cocharacter lattice $Y_0$ together with a quadratic refinement $q$ and an element of $\K_2(k)$,

\item $\mH^1(G, \cK_2)$ is isomorphic to the group of $W$-invariant quadratic forms on $Y_{der}$.
\end{itemize}

The higher cohomology groups vanish.
\end{thm}

\subsection{$\cK_3$-cohomology}

Recall that the $E_0$ page of the spectral sequence computing $\mH^\bullet(G, \cK_3)$ is
\[
\xymatrix@R-1pc@C-2pc{
\oplus_{W^{(3)}} \Z & 0 & 0 & 0 \\
\oplus_{W^{(2)}} X \ar[u] & 0 & 0 & 0 \\
\oplus_{W^{(1)}} \bigwedge^2 X \ar[u] & \oplus_{W^{(2)}} k^\times & 0 & 0\\
\bigwedge^3 X \ar[u] & \oplus_{W^{(1)}} X\otimes k^\times \ar[u] & 0 & 0 \\
0 & \bigwedge^2 X\otimes k^\times \ar[u] & \oplus_{W^{(1)}} \K_2(k) & 0 \\
0 & 0 & X\otimes \K_2(k) \ar[u] & 0 \\
0 & 0 & 0 & \K_3(k)
}
\]

As for the case of $\cK_2$-cohomology, we get $E_1^{-1, 1} = X_0\otimes \K_2(k)$ and $E_1^{-1, 2} = 0$.

The cohomology in the $(-2)$-nd column is \[E_1^{-2, 2} = \wedge^2 X_0\otimes k^\times, \quad E_1^{-2, 4} = 0\] and $E_1^{-2, 3}$ is the group of $W$-invariant quadratic forms on $Y_{der}$ with values in $k^\times$.

The cohomology group $E_1^{-3, 5}$ is the cohomology of the middle column in
\[
\xymatrix{
0 \ar[r] & \oplus_{W^{(1)}} X_{der}\otimes X_0 \ar[r] \ar[d] & \oplus_{W^{(1)}} \wedge^2 X / \wedge^2 X_0 \ar[r] \ar[d] & \oplus_{W^{(1)}} \wedge^2 X_{der} \ar[r] \ar[d] & 0\\
0 \ar[r] & \oplus_{W^{(2)}} X_0 \ar[r] \ar[d] & \oplus_{W^{(2)}} X \ar[r] \ar[d] & \oplus_{W^{(2)}} X_{der} \ar[r] \ar[dl] & 0 \\
& 0 & \oplus_{W^{(3)}} \Z & &
}
\]

The map
\[\bigoplus_{W^{(1)}} X_{der}\otimes X_0 \rightarrow \bigoplus_{W^{(2)}} X_0\]
is surjective (see Proposition \ref{prop:2rowsurjective}). Therefore, the cohomology of the middle column can be computed as the cohomology of the right-most column, which we have shown in the previous section to be isomorphic to the group of $W$-invariant cubic forms on $Y_{der}$.

The cohomology group $E_1^{-3, 4}$ is the cohomology of the middle column in
\[
\xymatrix{
0 \ar[r] & K \ar[r] \ar[d] & \wedge^3 X / \wedge^3 X_0 \ar[r] \ar[d] & \wedge^3 X_{der} \ar[r] \ar[d] & 0 \\
0 \ar[r] & \oplus_{W^{(1)}} X_{der}\otimes X_0 \ar[r] \ar[d] & \oplus_{W^{(1)}} \wedge^2 X / \wedge^2 X_0 \ar[r] \ar[d] & \oplus_{W^{(1)}} \wedge^2 X_{der} \ar[r] \ar[d] & 0\\
0 \ar[r] & \oplus_{W^{(2)}} X_0 \ar[r] & \oplus_{W^{(2)}} X \ar[r] & \oplus_{W^{(2)}} X_{der} \ar[r] & 0
}
\]

The kernel $K$ itself can be written as an extension
\[
0 \rightarrow X_{der}\otimes \wedge^2 X_0 \rightarrow K \rightarrow \wedge^2 X_{der}\otimes X_0\rightarrow 0.
\]

Note that the map $K\rightarrow \oplus_{W^{(1)}} X_{der}\otimes X_0$ factors through $K\rightarrow\wedge^2 X_{der}\otimes X_0$.

The cohomology of the right-most column is zero (Proposition \ref{prop:3colvanishing}). Therefore, $E_1^{-3, 4}$ is the cohomology of
\[\wedge^2 X_{der}\otimes X_0\rightarrow \oplus_{W^{(1)}} X_{der}\otimes X_0 \rightarrow \oplus_{W^{(2)}}X_0.\]

This complex is $X_0$ tensored with the $(-2)$-nd column of the $E_0$ page of the spectral sequence computing $\mH^\bullet(G, \cK_2)$.

Therefore, $E_1^{-3, 4}$ is isomorphic to the group of forms $Q(-, -)\colon Y_{der}\otimes Y_0$ quadratic and $W$-invariant in the first argument and linear in the second one.

The $E_1$ page of the spectral sequence is
\[
\xymatrix@R-2pc@C-2pc{
E_1^{-3,6} & 0 & 0 & 0 \\
E_1^{-3, 5} & 0 & 0 & 0 \\
E_1^{-3, 4} & 0 & 0 & 0\\
\wedge^3 X_0 & E_1^{-2, 3} & 0 & 0 \\
0 & \wedge^2 X_0\otimes k^\times & 0 & 0 \\
0 & 0 & X_0\otimes \K_2(k) & 0 \\
0 & 0 & 0 & \K_3(k)
}
\]

The only possible differential occurs at the $E_1$ page $\wedge^3 X_0\rightarrow E_1^{-2, 3}$. To see that it vanishes, observe that the map $\mH^0(H_0, \cK_3)\rightarrow \mH^0(G, \cK_3)$ induces an isomorphism on the diagonal terms of the $E_1$ pages and the differential vanishes for the $H_0$ spectral sequence.

We have a possible nontrivial extension
\[1\rightarrow E_1^{-2, 3}\rightarrow \mH^1(G, \cK_3)\rightarrow E_1^{-3, 4}\rightarrow 1\]
coming from the filtration on $\mH^0(H, \cK_2)$.

Recall (Proposition \ref{prop:k2extension}) that $\tilde{\mH}^0(H, \cK_2)$ is the group of alternating forms $A$ on $Y$ together with a quadratic refinement satisfying
\[q(x+y) = q(x)q(y) (-1)^{A(x, y)}.\]

The group $E_1^{-2, 3}$ has the forms $q$ restricted to $Y_{der}$. In contrast, the alternating forms $A$ appearing in $E_1^{-3, 4}$ vanish on $Y_{der}$. This shows the splitting of the exact sequence and hence we get the following result.

\begin{thm}
The $\cK_3$-cohomology of $G$ is as follows:
\begin{itemize}
\item $\mH^0(G, \cK_3)\cong \mH^0(H_0, \cK_3)$ is isomorphic to the direct sum of $\K_3(k)$ and the group of cubic forms on $Y_0$ together with refinements $q_1$, $q_2$ obeying the conditions of Proposition \ref{prop:k3extension},

\item $\mH^1(G, \cK_3)$ consists of quadratic forms on $Y_{der}$ valued in $k^\times$ and forms $Q(-, -)$ on $Y_{der}$ and $Y_0$ quadratic and $W$-invariant in the first argument and linear in the second,

\item $\mH^2(G, \cK_3)$ is isomorphic to the group of $W$-invariant cubic forms on $Y_{der}$,

\item $\mH^3(G, \cK_3)$ is isomorphic to $(\Z/2\Z)^{n_{\BDG}}$.
\end{itemize}

The higher cohomology groups vanish.
\end{thm}

\section{Central extensions of groups}

\label{sect:centralextensions}

\subsection{Central extensions}

Let $\B_\bullet G$ be the simplicial scheme whose space of $n$-simplices is $G^{\times n}$.

Consider a sheaf $\cF$ of abelian groups, such that $\mH^2(\Spec k, \cF) = 0$. In other words, every $\cF$-gerbe on the point is trivializable. Then $\mH^2(\B_\bullet G, \cF)$ is identified with the group of central extensions of $G$ by $\cF$.

Indeed, an element of $\mH^2(\B_\bullet G, \cF)$ is an $\cF$-gerbe on the point together with an $\cF$-torsor $\cT$ on $G$ and an isomorphism $\mathrm{mult}\colon m^*\cT\rightarrow p_1^*\cT\boxtimes p_2^*\cT$ on $G\times G$ satisfying an associativity condition on $G\times G\times G$. In other words, this is the data of a \textit{multiplicative torsor} on $G$.

Similarly, elements of $\mH^3(\B_\bullet G, \cF)$ are identified with isomorphism classes of \textit{multiplicative gerbes} on $G$. Recall that a multiplicative gerbe is an $\cF$-gerbe $\cG$ on $G$ together with an isomorphism of gerbes $\mathrm{mult}\colon m^*\cG\rightarrow p_1^*\cG\boxtimes p_2^*\cG$ on $G\times G$ and a natural transformation between two isomorphisms on $G\times G\times G$ obeying an associativity-like condition on $G^{\times 4}$.

The notion of a multiplicative gerbe is equivalent to that of a gerbal central extension of $G$ by $\B\cF$, but we will be content with the notion of a multiplicative gerbe.

\subsection{Cohomology of the classifying space}

Cohomology groups $\mH^\bullet(\B_\bullet G, \cF)$ can be computed via the spectral sequence of a double complex. Its $E_1$ page is
\[E_1^{p, q} = \mH^q(G^{\times p}, \cF)\]
with the $E_1$ differential $\mH^q(G^{\times p}, \cF)\rightarrow \mH^q(G^{\times(p+1)}, \cF)$ given by the alternating sum of the pullbacks along the face maps $G^{\times(p+1)}\rightarrow G^{\times p}$.

We will also need the following computation. Let $S^1_\bullet$ be the standard simplicial circle whose set of $n$-simplices consists of $n+1$ elements. Taking the alternating coface complex of the cosimplicial abelian group of functions $\Z[S^1_\bullet]$ on the circle, we get the complex
\[\mC^\bullet(S^1, \Z):=(\Z\rightarrow \Z^{\oplus 2}\rightarrow \Z^{\oplus 3} \rightarrow ...)\]
computing the cohomology $\mH^\bullet(S^1, \Z)$.

There is a subcomplex
\[\mC^\bullet(\pt, \Z):=(\Z\stackrel{0}\rightarrow \Z\stackrel{1}\rightarrow \Z\rightarrow ...)\]
of $\mC^\bullet(S^1, \Z)$ coming from the cohomology of the point and we denote by
\[\tilde{\mC}^\bullet(S^1, \Z) = (0\rightarrow \Z\rightarrow \Z^2\rightarrow ...)\]
the reduced complex. It computes the reduced cohomology $\tilde{\mH}^\bullet(S^1, \Z)$, which is $\Z$ concentrated in degree 1.

For any sheaf $\cF$ of abelian groups we have a map
\[\mH^m(G, \cF)^{\oplus n}\rightarrow \mH^m(G^{\times n}, \cF)\]
given by the sum of the pullbacks along projections. If we assume that these maps are isomorphisms for each $n$, the complex $\mH^m(\B_\bullet G, \cF)$ becomes isomorphic to $\mH^m(G, \cF)\otimes \tilde{\mC}^\bullet(S^1, \Z)$, which has cohomology $\mH^m(G, \cF)$ concentrated in degree 1.

\subsection{Central extensions of groups by $\cK_2$}

Let us compute the cohomology $\mH^\bullet(\B_\bullet G, \cK_2)$. The $E_1$ page of the spectral sequence is
\[
\xymatrix@R-1pc@C-1pc{
\mH^1(\pt, \cK_2) \ar[r] & \mH^1(G, \cK_2) \ar[r] & \mH^1(G\times G, \cK_2) \ar[r] & ...\\
\mH^0(\pt, \cK_2) \ar[r] & \mH^0(G, \cK_2) \ar[r] & \mH^0(G\times G, \cK_2) \ar[r] & ...
}
\]

Since there are no linear $W$-invariant forms on $Y_{der}$, the group of $W\times W$-invariant quadratic forms on $Y_{der}\oplus Y_{der}$ is the sum of the groups of $W$-invariant quadratic forms on each summand. This implies that the map
\[\mH^1(G, \cK_2)^{\oplus n}\rightarrow \mH^1(G^{\times n}, \cK_2)\]
introduced earlier is an isomorphism. Thus the first row of the spectral sequence forms the complex $\tilde{\mC}^\bullet(S^1, \Z)\otimes \mH^1(G, \cK_2)$, which has cohomology $\mH^1(G, \cK_2)$ in degree 1.

The zeroth row of the spectral sequence is isomorphic to the complex $\mH^0(\B_\bullet H_0, \cK_2)$. The cohomology of $\mH^0(\B_\bullet H_0, \cK_n)$ has been computed Esnault et al. (see \cite[Theorem 4.6]{EKLV}). Its cohomology in degree $m$ is
\[\mH^m \mH^0(\B_\bullet H_0, \cK_n) \cong \Sym^m(X_0)\otimes \K_{n-m}(k).\]

Therefore, the $E_2$ page of the spectral sequence is
\[
\xymatrix@R-2pc@C-2pc{
0 & \mH^1(G, \cK_2) & 0 & 0 \\
\K_2(k) & X_0\otimes k^\times & \Sym^2(X_0) & 0
}
\]

The spectral sequence degenerates at $E_2$. The following statement can be found in \cite[Theorem 4.7]{BrD}.

\begin{thm}[Brylinski--Deligne]
The cohomology $\mH^\bullet(\B_\bullet G, \cK_2)$ is given by
\begin{itemize}
\item $\mH^0(\B_\bullet G, \cK_2)\cong \K_2(k)$,

\item $\mH^1(\B_\bullet G, \cK_2)$ is isomorphic to the group of $W\!\text{-invariant}$ linear forms on the cocharacter lattice,

\item $\mH^2(\B_\bullet G, \cK_2)$ is isomorphic to the group of $W\!\text{-invariant}$ quadratic forms on the cocharacter lattice.
\end{itemize}

The cohomology vanishes in other degrees.
\end{thm}

\subsection{Central extensions of groups by $\cK_3$}

Let us repeat the arguments for the cohomology $\mH^\bullet(\B_\bullet G, \cK_3)$. In this case the $E_1$ page of the spectral sequence is
\[
\xymatrix@R-1pc@C-1pc{
\mH^3(\pt, \cK_3) \ar[r] & \mH^3(G, \cK_3) \ar[r] & \mH^3(G\times G, \cK_3) \ar[r] & ... \\
\mH^2(\pt, \cK_3) \ar[r] & \mH^2(G, \cK_3) \ar[r] & \mH^2(G\times G, \cK_3) \ar[r] & ...\\
\mH^1(\pt, \cK_3) \ar[r] & \mH^1(G, \cK_3) \ar[r] & \mH^1(G\times G, \cK_3) \ar[r] & ...\\
\mH^0(\pt, \cK_3) \ar[r] & \mH^0(H_0, \cK_3) \ar[r] & \mH^0(H_0\times T_0, \cK_3) \ar[r] & ...
}
\]

Since there are no linear $W$-invariant forms on $Y_{der}$, the $\mH^2$ row is additive, i.e.
\[E_1^{\bullet, 2} = \tilde{\mC}^\bullet(S^1, \Z)\otimes \mH^2(G, \cK_3).\]

The number of $\BDG$ factors in the decomposition of $G$ is also obviously additive, so the same complex occurs on the third row.

The first row splits as a direct sum of an additive complex representing $W$-invariant quadratic forms on $Y_{der}$ with values in $k^\times$ and a complex representing quadratic-linear forms on $Y_{der}$ and $Y_0$.

The complex computing quadratic-linear forms is quasi-isomorphic to the tensor product of two additive complexes by the Eilenberg-Zilber map. Therefore, it has cohomology concentrated in degree 2.

Therefore, the $E_2$ page of the spectral sequence is
\[
\xymatrix@R-2pc@C-1pc{
0 & \mH^3(G, \cK_3) & 0 & 0 \\
0 & \mH^2(G, \cK_3) & 0 & 0 \\
0 & E_2^{1, 1} & E_2^{2, 1} & 0 \\
\K_3(k) & X_0\otimes \K_2(k) & \Sym^2(X_0)\otimes k^\times & \Sym^3(X_0)
}
\]

There is a possible $E_2$ differential
\[E_2^{1, 1}\rightarrow \Sym^3(X_0).\]

To see that it is zero, consider the pullback map
\[\mH^\bullet(\B_\bullet H_0, \cK_3)\rightarrow \mH^\bullet(\B_\bullet G, \cK_3).\]

If a $\cK_3$-multiplicative gerbe is trivial on $G$, it is necessarily trivial on $H_0$ since $G\rightarrow H_0$ is surjective. Therefore, the pullback map is injective. In particular, $\Sym^3(X_0)\rightarrow E_\infty^{3, 0}$ is injective. But since $E_\infty^{3, 0}$ is a quotient of $\Sym^3(X_0)$, the $E_2$ differential vanishes and the spectral sequence degenerates at $E_2$.

\begin{thm}
The cohomology groups $\mH^\bullet(\B_\bullet G, \cK_3)$ are given as follows:
\begin{itemize}
\item $\mH^0(\B_\bullet G, \cK_3)\cong \K_3(k)$,

\item $\mH^1(\B_\bullet G, \cK_3)$ is isomorphic to the group of $W\!\text{-invariant}$ linear forms on the cocharacter lattice valued in $\K_2(k)$

\item $\mH^2(\B_\bullet G, \cK_3)$ is isomorphic to the group of $W\!\text{-invariant}$ quadratic forms on the cocharacter lattice valued in $k^\times$,

\item $\mH^3(\B_\bullet G, \cK_3)$ is isomorphic to the group of $W\!\text{-invariant}$ cubic forms on the cocharacter lattice,

\item $\mH^4(\B_\bullet G, \cK_3)$ is isomorphic to $(\Z/2\Z)^{n_{\BDG}}$.
\end{itemize}

The other cohomology groups vanish.

\label{thm:main}
\end{thm}

Using the fact that cohomology of the classifying space parametrizes central extensions, we conclude the following facts:
\begin{enumerate}
\item Central extensions of $G$ by $\cK_3$ are parametrized by $W$-invariant quadratic forms on the cocharacter lattice valued in $k^\times$.

\item Gerbal central extensions of $G$ by $\cK_3$ are parametrized by $W$-invariant cubic forms on the cocharacter lattice.
\end{enumerate}

\subsection{Extensions of loop groups}

Let us explain how thus obtained central extensions of groups by $\cK_2$ and $\cK_3$ relate to central extensions of loop and double loop groups. This is an algebraic counterpart \cite[Section 4.3]{KV} of transgression of central extensions to smooth loop groups, see e.g. \cite[Section 6]{Br}. See also \cite{Wal} for an explicit relation between multiplicative gerbes on $G$ and central extensions of the loop group.

For a space $X$ over a field $k$ (i.e. a sheaf over the category of affine schemes over $k$) we denote
\[\L X = \Map(\Spec\ k\llpar t\rrpar, X)\]
the loop space of $X$.

Consider a correspondence
\[
\xymatrix{
& \L X\times D^\times \ar_{\pi}[dl] \ar^{\ev}[dr] & \\
\L X && X
}
\]

Given an element $t\in \mH^n(X, \cK_m)$ its pullback $\ev^* t$ defines an element of $\mH^n(\L X\times D^\times, \cK_m)$.

The residue map on $\K$-theory defines a map
\[\pi_*\colon\mH^n(\L X\times D^\times, \cK_m)\rightarrow \mH^n(\L X, \cK_{m-1}).\]

Thus, we have defined a transgression map
\[\pi_* \ev^*\colon \mH^n(X, \cK_m)\rightarrow \mH^n(\L X, \cK_{m-1})\]
which sends an $(n-1)$-gerbe over $\cK_m$ on $X$ to an $(n-1)$-gerbe over $\cK_{m-1}$ on $\L X$.

This behaves well with respect to multiplicative structures. Indeed,
\[\L\B_\bullet G\cong \B_\bullet \L G.\]

Therefore, we get a map
\[\mH^{n+1}(\B_\bullet G, \cK_m)\rightarrow \mH^{n+1}(\B_\bullet \L G, \cK_{m-1})\]
which sends multiplicative $(n-1)$-gerbe over $\cK_m$ on $G$ to a multiplicative $(n-1)$-gerbe over $\cK_{m-1}$ on $\L G$.

For instance, a multiplicative $\cK_2$-torsor on $G$ gives rise to a $\G_m$-central extension of $\L G$. In this way we get a central extension of $\L G$ by $\G_m$ for any $W$-invariant quadratic form on the cocharacter lattice.

The results of this paper produce a $\B\G_m$-central extension of $\L^2G\equiv \L\L G$, the double loop group, given any $W$-invariant cubic form on the cocharacter lattice. Similarly, we get a $\G_m$-central extension of $\L^2 G$ given any $W$-invariant quadratic form on the cocharacter lattice valued in $k^\times$. We expect that the gerbal central extension of $\L^2\mathrm{GL}_n$ constructed in \cite{AK} and \cite{FZ} come from the basic $W$-invariant cubic form.

\end{document}